\documentclass[10pt]{amsart}

\usepackage{float}
\usepackage{amscd,amsmath,amssymb,amsfonts,amsthm}
\usepackage{enumerate,mathrsfs,stmaryrd,latexsym, comment, mathtools, mathdots} 
\usepackage[mathscr]{euscript}

\usepackage{amsthm, amssymb}
\usepackage[leqno]{amsmath}
\usepackage{caption}
\usepackage{subcaption}
\usepackage{hyperref}
\usepackage{relsize}
\usepackage{booktabs}
\usepackage{a4wide}

\usepackage{tikz-cd}
\usepackage{tikz}
\usetikzlibrary{snakes, 
	3d, matrix,decorations.pathreplacing,calc,decorations.pathmorphing, 
	patterns,automata}
\usetikzlibrary {positioning}
\usepackage{tikz-qtree}
\usetikzlibrary{shapes.geometric}

\numberwithin{equation}{section}
\allowdisplaybreaks[1]

\allowdisplaybreaks

\newcommand{\C}{{\mathbb{C}}}

\newcommand{\Z}{{\mathbb{Z}}}

\def\SF{\mathscr{SF}}
\def\ST{\mathscr{ST}}

\def\F{\mathscr{F}}
\newcommand{\T}{\mathcal{T}}
\newcommand{\s}{\mathsf{s}}
\newcommand{\mcal}{\mathcal}

\def\f{t}

\def\MT{\mathscr{T}}
\def\MF{\mathscr{F}}

\newcommand{\Bt}{\mathscr{B}}

\DeclareMathOperator{\Cone}{Cone}
\DeclareMathOperator{\PC}{PC}
\DeclareMathOperator{\GL}{GL}

\newtheorem{theorem}{Theorem}[section]
\newtheorem{lemma}[theorem]{Lemma}
\newtheorem{proposition}[theorem]{Proposition}
\newtheorem{corollary}[theorem]{Corollary}

\newtheorem{Question}[theorem]{Question}

\theoremstyle{definition}
\newtheorem{example}[theorem]{Example}
\newtheorem{definition}[theorem]{Definition}
\newtheorem{remark}[theorem]{Remark}

\hypersetup{colorlinks}
\hypersetup{
	unicode=false,          
	pdftoolbar=true,        
	pdfmenubar=true,        
	pdffitwindow=false,     
	pdfstartview={FitH},    
	pdftitle={My title},    
	pdfauthor={Author},     
	pdfsubject={Subject},   
	pdfcreator={Creator},   
	pdfproducer={Producer}, 
	pdfkeywords={keyword1} {key2} {key3}, 
	pdfnewwindow=true,      
	colorlinks=true,       
	linkcolor=blue,          
	citecolor=red,        
	filecolor=violet,      
	urlcolor=violet       
}

\tikzset{root/.style = {circle, double, draw, inner sep = 1pt}}
\tikzset{vertex/.style = {circle, fill, inner sep = 1.5pt}}

\begin{document}

\author{Yunhyung Cho}
\address{Department of Mathematics Education, Sungkyunkwan University, 
Seoul, Republic of Korea}
\email{yunhyung@skku.edu}

\author{Eunjeong Lee}
\address{Center for Geometry and Physics, Institute for Basic Science (IBS), 
Pohang 37673, Republic of Korea}
\email{eunjeong.lee@ibs.re.kr}

\author{Mikiya Masuda}
\address{Osaka City University Advanced Mathematics Institute (OCAMI) \& 
Department of Mathematics, Graduate School of Science, Osaka City University, 
Sumiyoshi-ku, Sugimoto, 558-8585, Osaka, Japan}
\email{mikiyamsd@gmail.com}

\author{Seonjeong Park}
\address{Department of Mathematics Education, Jeonju University, Jeonju 55069, 
Republic of Korea}
\email{seonjeongpark@jj.ac.kr}



\title{On the enumeration of Fano Bott manifolds}

\date{\today}

\subjclass[2020]{Primary: 14M25, secondary: 05C22}

\keywords{Bott manifolds, Husimi trees}

\maketitle

\begin{abstract} 
Fano Bott manifolds bijectively correspond to signed rooted 
forests with some equivalence relation. Using this bijective correspondence, we 
enumerate the isomorphism classes of Fano Bott manifolds and the diffeomorphism 
classes of indecomposable Fano Bott manifolds. We also observe that the signed 
rooted forests with the equivalence relation bijectively correspond to rooted 
triangular cacti.
\end{abstract}

\setcounter{tocdepth}{1} \tableofcontents
\section{Introduction}

A Bott manifold of complex dimension $n$ is a smooth projective toric variety 
whose fan 
is the normal fan of a polytope combinatorially equivalent to the cube 
$[0,1]^n$. A family of Bott manifolds was first considered by Grossberg and 
Karshon~\cite{GK94Bott} in the context of toric degenerations of 
Bott--Samelson varieties. Since then, topological or geometric 
properties of Bott manifolds have been intensively studied in~\cite{ChoiSuh11, ChoiMasudaMurai, BCT19_Kahler, Chary18}. 
Recently, motivated by Suyama's work~\cite{Suyama20}, we showed the 
$c_1$-cohomological rigidity for Fano Bott manifolds, which means that two Fano 
Bott manifolds are isomorphic if and only if there is a cohomology ring 
isomorphism between them preserving their first Chern classes (\cite{CLMP}). 

It is known that there are only finitely many smooth Fano toric varieties up to 
isomorphism in each dimension (cf.~\cite{Obro07}), and 
therefore there are also only finitely 
many Fano Bott manifolds up to isomorphism in each dimension.
Higashitani and Kurimoto~\cite{HigashitaniKurimoto20} associate \emph{signed 
rooted forests} with Fano Bott manifolds to classify Fano Bott manifolds up to diffeomorphism. In this paper, we enumerate the 
isomorphism classes of Fano Bott manifolds and the diffeomorphism classes of 
indecomposable Fano Bott manifolds using this correspondence.

To introduce our main result, we prepare some terminologies.
Recall that a fan associated to a Bott manifold of complex dimension $n$ is the normal fan of a polytope combinatorially equivalent to $[0,1]^n$ and so it has $2n$ 
rays. We denote the primitive ray generators by 
$\mathbf{v}_1,\dots,\mathbf{v}_n, \mathbf{w}_1,\dots,\mathbf{w}_n$, where 
$\mathbf{v}_i$ and $\mathbf{w}_i$ are pairwise normal vectors of {\em opposite facets}. 

If a Bott manifold is {\em Fano}, then it is known that the sum $\mathbf{v}_i + \mathbf{w}_i$ is either the zero vector or another ray generator, say $\mathbf{v}_{\varphi(i)}$ or $\mathbf{w}_{\varphi(i)}$, where $\phi$ is a permutation on $[n]$.
Accordingly, one may associate a signed rooted forest with a Fano Bott 
manifold as follows: the vertices are $[n]=\{1,\dots,n\}$, 
the parent of $i$ is $\varphi(i)$, and the edge $\{i,\varphi(i)\}$ is signed by~$+$ if $\mathbf{v}_i 
+ \mathbf{w}_i =\mathbf{v}_{\varphi(i)}$; and by $-$ otherwise. Indeed, whenever the sum $\mathbf{v}_i + \mathbf{w}_i$ is the zero vector, the vertex $i$ is a root (see 
Section~\ref{sec_FB_signed_rooted_forests} for more precise definition).

For each vertex $i$ of a signed rooted forest, we obtain another signed rooted 
forest by changing the signs of all the edges connecting $i$ and its children 
simultaneously. By considering this operation for all vertices, we obtain an 
equivalence relation~$\sim$ on the isomorphism classes $\SF_n$ of signed rooted 
forests with vertices $[n]$. 
It is observed in~\cite[Remark~5.8]{HigashitaniKurimoto20} that the isomorphism classes in Fano Bott manifolds of complex dimension $n$ bijectively correspond to the equivalence classes $\SF_n/\!\!\sim$ of signed rooted forests with $n$ vertices.
Now we state our main theorem. 
\begin{theorem}[{Corollary~\ref{cor_generating_ftn_of_Fn}}]
The generating function 
$\F(x)=\sum_{n=0}^\infty |\SF_n/\!\!\sim| x^n$ satisfies
\[
\MF(x)=\exp\left(\sum_{k=1}^\infty\frac{x^k}{2k}
\left(\MF(x^{2k})+\MF(x^k)^2\right)\right).
\]
\end{theorem}

This functional equation determines $\MF(x)$. Indeed, a straightforward 
computation shows 
\[
\MF(x)=1+x+2x^2+5x^3+13x^4+37x^5+111x^6+345x^7+1105x^8+3624x^9+\cdots
\]
The generating function $\Delta(x)=1+\sum_{n=1}^\infty\Delta_nx^n$ for the number $\Delta_n$ of rooted triangular cacti with $2n+1$ vertices and $n$ triangles satisfies the same functional equation (\cite{HaNo53}, \cite{HaUh53}). It turns out that there is a bijective correspondence between $\SF_n/\!\!\sim$ and rooted triangular cacti with $2n+1$ vertices and $n$ triangles.

The result~\cite{HigashitaniKurimoto20} by Higashitani and 
Kurimoto implies that the diffeomorphism classes of indecomposable Fano Bott 
manifolds of complex dimension $n$ bijectively correspond to 
$\SF_{n-1}/\!\!\sim$. 
Here, we say a Fano Bott manifold is \emph{indecomposable} if it is not 
isomorphic to a product of lower dimensional Fano Bott manifolds.
This provides an enumeration of the diffeomorphism classes of indecomposable 
Fano Bott manifolds.

This paper is organized as follows. 
In Section~\ref{sec_FB_signed_rooted_forests}, we provide the definition of 
Bott manifolds and their Fano conditions. Moreover, we recall the association 
of signed rooted forests with Fano Bott manifolds.
In Section~\ref{sec_classification}, 
we show that the association induces a bijection between the isomorphism classes in Fano Bott manifolds and the equivalence classes $\SF_n/\!\!\sim$ of signed rooted forests with $n$ vertices.  In Section~\ref{sec_counting}, we enumerate the equivalence classes $\SF_n/\!\!\sim$ of signed 
rooted forests. In Section~\ref{sec_cacti}, we give a 
bijective correspondence between $\SF_n/\!\!\sim$ and rooted triangular cacti 
with $2n+1$ vertices and $n$ triangles.

\subsection*{Acknowledgments}
Cho was supported by the National Research Foundation of Korea(NRF) grant 
funded by the Korea government(MSIP; Ministry of Science, ICT \& Future 
Planning) (No.2020R1C1C1A01010972) and (No.2020R1A5A1016126).
Lee was supported by the Institute for Basic Science (IBS-R003-D1). Masuda was 
supported in part by JSPS Grant-in-Aid for Scientific Research 19K03472 and a 
HSE University Basic Research Program. Park was supported by the Basic Science 
Research Program through the National Research Foundation of Korea (NRF) 
funded by the Government of Korea (NRF-2018R1A6A3A11047606). This work 
was partly supported by Osaka City University Advanced Mathematical Institute 
(MEXT Joint Usage/Research Center on Mathematics and Theoretical Physics 
JPMXP0619217849).
\section{Fano Bott manifolds and signed rooted forests}
\label{sec_FB_signed_rooted_forests}

In this section we review the definition of Bott manifolds and their fans. We 
also recall the relation between Fano Bott manifolds and signed rooted forests 
from~\cite{HigashitaniKurimoto20}.
\begin{definition}[{\cite[\S 2.1]{GK94Bott}}]
A Bott tower $\Bt_{\bullet}$ is an iterated $\C P^1$-bundle starting with a point:
\[
\begin{tikzcd}[row sep = 0.6em]
\Bt_n \rar & \Bt_{n-1} \rar & \cdots \rar & \Bt_1 \rar & \Bt_0, \\
P(\underline{\C} \oplus \xi_n) \arrow[u, equal]& & & \C P^1 \arrow[u, equal] & 
\{\text{a point}\} \arrow[u, equal]
\end{tikzcd}
\]
where each $\Bt_i$ is the complex projectivization of the Whitney sum of a 
holomorphic line bundle~$\xi_i$ and the trivial line bundle $\underline{\C}$ 
over $\Bt_{i-1}$. The total space $\Bt_n$ is called a \textit{Bott manifold}. 
\end{definition}

A Bott manifold $\Bt_n$ is a smooth projective toric variety by the construction. 
Its fan $\Sigma$ has $2n$ rays. We denote by $\{ \mathbf{v}_1, \dots, 
\mathbf{v}_n, \mathbf{w}_1,\dots,\mathbf{w}_n\}$ the ray generators, where a 
pair of $\mathbf{v}_i$ and $\mathbf{w}_i$ does not span a cone for each $i$. A 
subset $S$ of ray generators having $n$ elements form a maximal cone of 
$\Sigma$ if and only if 
\[
\{ \mathbf{v}_i, \mathbf{w}_i \} \not\subset S \quad \text{ for any }i \in [n].
\]
Because of this description, one may see that the fan $\Sigma$ is the normal fan 
of a polytope combinatorially equivalent to the cube $[0,1]^n$. 

For a fan $\Sigma$ and its ray generators $\{ \mathbf u_{\rho} \mid \rho \in 
\Sigma(1)\}$ in $\Sigma$, we call a subset $P \subset \{ \mathbf u_{\rho}\mid 
\rho \in \Sigma(1)\}$ a \emph{primitive collection} if 
\[
\Cone(P) \notin \Sigma
\quad \text{ but }\Cone(P \setminus \{x\}) \in \Sigma \quad \text{ for every }x 
\in P.
\] 
We denote by $\PC(\Sigma)$ the set of primitive collections of $\Sigma$. 
We briefly review Batyrev's criterion~\cite[Proposition~2.3.6]{Batyrev}, which 
determines whether a given toric variety is Fano or not. Let $\Sigma$ be a 
smooth complete fan. For each primitive collection $P = \{ 
\mathbf{u}_1,\dots,\mathbf{u}_{r}\}$, there exists a unique cone $\sigma$ such 
that $\mathbf{u}_1 + \cdots + \mathbf{u}_r$ is in the relative interior of 
$\sigma$. Let $\mathbf{v}_1,\dots,\mathbf{v}_{\ell}$ be the primitive generators 
of $\sigma$. Then
\[
\mathbf{u}_1 + \cdots + \mathbf{u}_r = a_1 \mathbf {v}_1 + \cdots + a_{\ell} 
\mathbf {v}_{\ell}
\]
for some positive integers $a_1,\dots,a_{\ell}$. 
If the sum of primitive generators is the zero vector, then the cone $\sigma$ is of zero-dimensional and the set $\{\mathbf{v}_1,\dots,\mathbf{v}_{\ell}\}$ is assumed to be empty.
We call this relation the 
\textit{primitive relation} for $P$ and we define the \textit{degree} of $P$ by
\[
\deg(P) \coloneqq r - (a_1 + \cdots + a_{\ell}).
\]
Here, we note that if the sum of primitive generators of $P$ is the zero vector, then $\deg(P) = r$.
\begin{proposition}[{\cite[Proposition~2.3.6]{Batyrev}}]\label{Prop_Batyrev}
Let $X_{\Sigma}$ be a nonsingular projective toric variety and $\PC(\Sigma)$ 
be the primitive collection of the fan $\Sigma$. Then 
the toric variety $X_{\Sigma}$ is Fano if and only if $\deg(P) > 0$ for 
every $P \in \PC(\Sigma)$.
\end{proposition}

Now we apply Batyrev's criterion to Bott manifolds. 
Let $\Sigma$ be the fan of a Bott manifold~$\Bt_n$. Then the set of primitive 
collection is 
\begin{equation}\label{eq_primitive_collection}
\PC(\Sigma) = \{ \{ \mathbf{v}_i, \mathbf{w}_i \} \mid i \in [n] \}.
\end{equation}
Using Proposition~\ref{Prop_Batyrev}, we can see that $\Bt_n$ is Fano if and only 
if each 
primitive collection $P = \{\mathbf{v}_i, \mathbf{w}_i\}$ satisfies one of the 
following:
\begin{enumerate}
\item $\mathbf{v}_i + \mathbf{w}_i = \mathbf{0}$ (that is, $\deg(P) = 2 > 
0$); 
\item $\mathbf{v}_i + \mathbf{w}_i = \mathbf{v}_{\varphi(i)}$ (that is, 
$\deg(P) = 2 -1 = 1> 0$); or
\item $\mathbf{v}_i + \mathbf{w}_i = \mathbf{w}_{\varphi(i)}$ (that is, 
$\deg(P) = 2 -1 =1 > 0$).
\end{enumerate}
Here, $\varphi \colon [n]\setminus Z \to [n]$, where $Z \coloneqq \{ i \mid 
\mathbf{v}_i + \mathbf{w}_i = \mathbf{0}\}$.
We also define a \emph{sign map} $\sigma \colon [n]\setminus Z \to \{ +, -\}$ 
by 
\[
\sigma(i) = \begin{cases}
+ & \text{ if } \mathbf{v}_i + \mathbf{w}_i = \mathbf{v}_{\varphi(i)}, \\
- & \text{ if } \mathbf{v}_i + \mathbf{w}_i = \mathbf{w}_{\varphi(i)}.
\end{cases}
\]
This leads us to the following definition. 

\begin{definition}[{\cite[Definition~4.1]{HigashitaniKurimoto20}}]
\label{def_SRF}
Let $\Sigma$ be the fan of a Fano Bott manifold
having (ordered) ray generators 
$\mathcal{S} = (\mathbf 
v_1,\dots,\mathbf{v}_n,\mathbf{w}_1,\dots,\mathbf{w}_n)$ with the primitive 
collections as in~\eqref{eq_primitive_collection}. 
Let $\varphi$ and $\sigma$ be as above. 
We define the \textit{associated signed rooted forest $(\T, \s) = 
(\T(\Sigma, \mathcal{S}), \s(\Sigma, \mathcal{S}))$} (i.e., rooted 
forest $\T$ with the sign map $\s \colon E(\T) \to \{ +, -\}$) to be
\begin{itemize}
\item $V(\T) = [n]$;
\item $E(\T) = \left\{ \left\{i, \varphi(i) \right\} \mid i \in [n] 
\setminus Z 
\right\}$ and $\s(\{i, \varphi(i)\}) = \sigma(i)$.
\end{itemize}
\end{definition} 
From the definition, one can see that for a singed rooted forest $(\T, \s)$, 
the set of roots is $Z$ and the parent of each vertex $i \in [n] \setminus Z$ 
is $\varphi(i)$. We denote the assignment provided in 
Definition~\ref{def_SRF} by $\Phi$, that is, $\Phi(\Sigma, \mathcal{S}) 
= (\T(\Sigma, \mathcal{S}), \s(\Sigma, \mathcal{S}))$ is the associated signed 
rooted forest.
\begin{remark}\label{rmk_surjectivity}
The association $\Phi$ is surjective, that is, for each signed rooted forest 
$(\T,\s)$ with $n$ vertices, there exists a Fano Bott manifold of 
dimension~$n$ whose fan defines $(\T,\s)$. 
\end{remark}
\begin{example}
We will present ray generators of the fan of a Bott manifold using a matrix, 
i.e., the columns of an $n \times 2n$ matrix are ray generators. 
Consider the following two matrices.
\[
A = \left[
\begin{array}{ccc|ccc}
1 & 0 & 0 & -1 & 0 & 0 \\
1 & 1 & 0 & 0 & -1 & -1 \\
0 & 0 & 1 & 0 & 0 & -1
\end{array}\right], \quad
A' = \left[
\begin{array}{ccc|ccc}
1 & 0 & 0 & -1 & 0 & 0 \\
1 & 1 & 0 & 0 & -1 & 0 \\
1 & 0 & 1 & 0 & 1 & -1
\end{array}\right].
\]
Let $\Bt$ be the Bott manifold such that the ray generators $(\mathbf 
v_1,\mathbf 
v_2, \mathbf v_3, \mathbf w_1, \mathbf w_2, \mathbf w_3)$ of the fan are the 
column vectors of $A$. Then, we have 
\[
\begin{split}
&\mathbf{v}_1 + \mathbf{w}_1 = \mathbf v_2, \\
&\mathbf{v}_2 + \mathbf{w}_2 = \mathbf{0}, \\
&\mathbf{v}_3 + \mathbf{w}_3 = \mathbf w_2.
\end{split}
\]
Therefore, the Bott manifold $\Bt$ is Fano, and moreover, $\varphi(1) = 2, 
\varphi(3) = 2$, and $\sigma(1) = +$, $\sigma(3) = -$. 
The associated signed rooted tree is given in 
Figure~\ref{fig_signed_rooted_tree_n3_6} (without vertex labeling).

Let $\Bt'$ be the Bott manifold such that ray generators $(\mathbf v_1,\mathbf 
v_2, \mathbf v_3, \mathbf w_1, \mathbf w_2, \mathbf w_3)$ of the fan are the 
column vectors of $A'$. 
Consider a primitive collection $P = \{ \mathbf{v}_1, \mathbf{w}_1\}$. The sum 
of ray generators is 
\[
\mathbf{v}_1 + \mathbf{w}_1 = \mathbf{v}_2 + \mathbf{v}_3,
\]
so $\deg(P) = 2 -2 = 0 \not> 0$ and this primitive collection does not satisfy the 
Fano condition. Therefore, the Bott manifold $\Bt'$ is not Fano.
\end{example}
We provide all signed rooted forests having three vertices in 
Figure~\ref{fig_SF_3}.
\begin{figure}[h]
\begin{subfigure}[b]{0.12\textwidth}
\centering
\begin{tikzpicture}[ node distance = 1.5em]

\node[root, vertex] (1) {};
\node[vertex, below = of 1] (2) {};
\node[vertex, below = of 2] (3) {};
\draw (1)--(2) node[midway, left] {$+$};
\draw (2)--(3) node[midway, left] {$+$};
\end{tikzpicture}
\caption{$(\T_1,\s_1)$} \label{fig_signed_rooted_tree_n3_1}
\end{subfigure}
\begin{subfigure}[b]{0.12\textwidth}
\centering
\begin{tikzpicture}[ node distance = 1.5em]
\node[root, vertex] (1) {};
\node[vertex, below = of 1] (2) {};
\node[vertex, below = of 2] (3) {};
\draw (1)--(2) node[midway, left] {$+$};
\draw (2)--(3) node[midway, left] {$-$};
\end{tikzpicture}
\caption{$(\T_2,\s_2)$} \label{fig_signed_rooted_tree_n3_2}
\end{subfigure}
\begin{subfigure}[b]{0.12\textwidth}
\centering
\begin{tikzpicture}[ node distance = 1.5em]
\node[root, vertex] (1) {};
\node[vertex, below = of 1] (2) {};
\node[vertex, below = of 2] (3) {};
\draw (1)--(2) node[midway, left] {$-$};
\draw (2)--(3) node[midway, left] {$+$};
\end{tikzpicture}
\caption{$(\T_3,\s_3)$} \label{fig_signed_rooted_tree_n3_3}
\end{subfigure}
\begin{subfigure}[b]{0.12\textwidth}
\centering
\begin{tikzpicture}[ node distance = 1.5em]
\node[root, vertex] (1) {};
\node[vertex, below = of 1] (2) {};
\node[vertex, below = of 2] (3) {};
\draw (1)--(2) node[midway, left] {$-$};
\draw (2)--(3) node[midway, left] {$-$};
\end{tikzpicture}
\caption{$(\T_4,\s_4)$} \label{fig_signed_rooted_tree_n3_4}
\end{subfigure}
\begin{subfigure}[b]{0.14\textwidth}
\centering
\raisebox{2em}{ \begin{tikzpicture}[node distance = 1.5em and 1 em]
\node[root, vertex] (1) {};
\node[vertex, below left = of 1] (2) {};
\node[vertex, below right = of 1] (3) {};
\draw (1)--(2) node[midway, left] {$+$};
\draw (1)--(3) node[midway, right] {$+$};
\end{tikzpicture}}
\caption{$(\T_5,\s_5)$} \label{fig_signed_rooted_tree_n3_5}
\end{subfigure}
\begin{subfigure}[b]{0.14\textwidth}
\centering
\raisebox{2em}{ \begin{tikzpicture}[node distance = 1.5em and 1 em]
\node[root, vertex] (1) {};
\node[vertex, below left = of 1] (2) {};
\node[vertex, below right = of 1] (3) {};
\draw (1)--(2) node[midway, left] {$+$};
\draw (1)--(3) node[midway, right] {$-$};
\end{tikzpicture}}
\caption{$(\T_6,\s_6)$} \label{fig_signed_rooted_tree_n3_6}
\end{subfigure} 
\begin{subfigure}[b]{0.14\textwidth}
\centering
\raisebox{2em}{ \begin{tikzpicture}[node distance = 1.5em and 1 em]
\node[root, vertex] (1) {};
\node[vertex, below left = of 1] (2) {};
\node[vertex, below right = of 1] (3) {};
\draw (1)--(2) node[midway, left] {$-$};
\draw (1)--(3) node[midway, right] {$-$};
\end{tikzpicture}}
\caption{$(\T_7,\s_7)$} \label{fig_signed_rooted_tree_n3_7}
\end{subfigure} 

\vspace{1em}

\begin{subfigure}[b]{0.16\textwidth}
\centering
\begin{tikzpicture}[ node distance = 1.5em]
\node[root, vertex] (1) {};
\node[vertex, below = of 1] (2) {};
\node[root,vertex, right = of 1] (3) {};
\draw (1)--(2) node[midway, left] {$+$};
\end{tikzpicture}
\caption{$(\T_8,\s_8)$} \label{fig_signed_rooted_tree_n3_8}
\end{subfigure} 
\begin{subfigure}[b]{0.16\textwidth}
\centering
\begin{tikzpicture}[ node distance = 1.5em]
\node[root, vertex] (1) {};
\node[vertex, below = of 1] (2) {};
\node[root, vertex, right = of 1] (3) {};

\draw (1)--(2) node[midway, left] {$-$};
\end{tikzpicture}
\caption{$(\T_9,\s_9)$} \label{fig_signed_rooted_tree_n3_9}
\end{subfigure}
\begin{subfigure}[b]{0.2\textwidth}
\centering
\raisebox{2em}{ \begin{tikzpicture}[ node distance = 1.5em]
\node[root, vertex] (1) {};
\node[root, vertex, right = of 1] (2) {};
\node[root, vertex, right = of 2] (3) {};
\end{tikzpicture}}
\caption{$(\T_{10},\s_{10})$} \label{fig_signed_rooted_tree_n3_10}
\end{subfigure}
\caption{Signed rooted forests with $3$ 
vertices.}\label{fig_SF_3}
\end{figure}
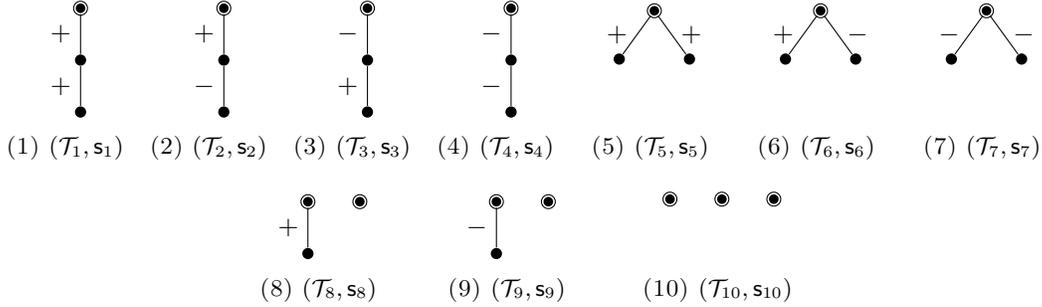

\begin{remark}
We say that a signed rooted forest is \emph{binary} if each vertex has at most 
two children and when the vertex has two children, the edges connecting the 
vertex and its children have different signs. In Figure~\ref{fig_SF_3}, all but 
(5) and (7) are binary. Binary rooted forests appear in connection with a 
certain family of Fano toric Richardson varieties (called \emph{of Catalan 
type}) in the full flag variety (\cite{LMP21}).
\end{remark}

\section{Classification of Fano Bott manifolds} 
\label{sec_classification}
We say that signed rooted forests $(\T,\s)$ and $(\T',\s')$ with vertices $[n]$ 
are \emph{isomorphic}
if there is a permutation $\pi \in 
\mathfrak{S}_n$ which sends the roots of $\T$ to the roots of $\T'$ and induces 
a bijection between the edges preserving the signs. 
Let $\SF_n$ be the isomorphism classes of signed rooted forests with vertices 
$[n]$. For each vertex $i \in [n]$, we define an operation
\[
r_i : \SF_n \rightarrow \SF_n
\]
which changes the signs of all edges connecting the vertex $i$ and its children 
simultaneously. 
Denote by $\sim$ the equivalence relation on $\SF_n$ generated by the 
operations $r_i$ for all $i \in [n]$. 
The following is mentioned in~\cite[Remark~5.8]{HigashitaniKurimoto20}, but we include its proof for readers' convenience.

\begin{theorem}[{cf.~\cite[Remark~5.8]{HigashitaniKurimoto20}}]\label{thm_isom_FB}
The isomorphism classes in Fano Bott manifolds of complex dimension $n$ 
bijectively correspond to $\SF_n/\!\!\sim$.
\end{theorem}
\begin{proof}
Let $\Bt$ be a Fano Bott manifold of complex dimension $n$,
$\Sigma = \Sigma_{\Bt}$ a fan defining~$\Bt$. 
We fix an ordering on ray generators $\mathcal S = (\mathbf 
v_1,\dots,\mathbf v_n, \mathbf w_1,\dots,\mathbf w_n)$ of $\Sigma$. 
For $A \in \GL(n,\Z)$, we denote by $A\cdot \Sigma$ the fan consisting of cones 
$A\cdot \sigma$'s for $\sigma \in \Sigma$ and denote by $A\cdot \mcal{S}$ the 
ordered ray generators of $A \cdot \Sigma$ given by 
\[
A \cdot \mcal{S} \coloneqq (A \mathbf{v}_1,\dots, A \mathbf{v}_n,
A \mathbf w_1,\dots,A \mathbf w_n).
\]

By~\cite[Proposition~3.4]{CLMP}, any other pair $(\Sigma', \mathcal{S}')$ 
defines a Fano Bott manifold isomorphic to $\Bt$ if and only if 
$\Sigma' = A \cdot \Sigma_{\Bt}$ for some $A \in \GL(n,\Z)$ and 
the set $\mathcal{S}'$ is obtained from $A \cdot \mathcal{S}$ by performing 
the following two operations on $A \cdot \mcal{S}$:
\begin{itemize}
\item[({\bf Op.1})] swapping $A{\bf v}_i$ with $A{\bf w}_i$, that is, 
\[
\mathcal{S}_i' \coloneqq (A\mathbf v_1,\dots,A\mathbf 
v_{i-1},A\mathbf{w}_i,A\mathbf{v}_{i+1},\dots,A\mathbf{v}_n,
A\mathbf{w}_1,\dots,A\mathbf{w}_{i-1},A\mathbf{v}_i,A\mathbf{w}_{i+1},\dots,
A\mathbf{w}_n);
\]
\item[({\bf Op.2})] reordering $A{\bf v}_{i}$ (as well as $A{\bf w}_i$'s), 
that is, for a permutation $\pi \in \mathfrak{S}_n$, 
\[
\mathcal{S}'_{\pi} \coloneqq (A\mathbf v_{\pi(1)},\dots,A\mathbf v_{\pi(n)},
A\mathbf w_{\pi(1)},\dots,A \mathbf w_{\pi(n)}).
\]
\end{itemize}
For the ordered ray generators $\mathcal{S}'_i$ obtained by applying 
({\bf Op.1}), we have 
\(
\Phi(\Sigma', \mathcal{S}_i') = r_i(\Phi(\Sigma_{\Bt},\mathcal{S})).
\)
For the ordered ray generators $\mcal{S}'_{\pi}$ obtained by applying 
({\bf Op.2}), 
$\Phi(\Sigma',\mcal{S}'_{\pi})$ is obtained from $\Phi(\Sigma_{\Bt}, \mcal{S})$ 
by changing the numbering of the vertices by the permutation $\pi$, so they
are isomorphic as signed rooted forests.
This finishes the proof. 
\end{proof}
\begin{example}
Consider $\SF_3$ described in Figure~\ref{fig_SF_3}. Then we obtain the five 
equivalence classes
\begin{gather*}
(\T_1,\s_1)\sim(\T_2,\s_2)\sim(\T_3,\s_3)\sim(\T_4,\s_4), \quad 
(\T_5,\s_5)\sim(\T_7,\s_7), \\
(\T_6,\s_6),\quad (\T_8,\s_8)\sim(\T_9,\s_9), 
\quad 
(\T_{10},\s_{10}).
\end{gather*}
All signed rooted forests in $\SF_4$ are illustrated in Figure 
\ref{fig_SF4}.
Roots of the forests are the top vertices. We omit plus signs on edges and put 
a minus sign on an edge. 
We also write ID numbers of the corresponding Fano Bott manifolds according to the list of `Smooth toric Fano 
varieties'~\cite{Obro07} in the Graded Ring Database~\cite{GRDB}.

\begin{figure}[hbtp]
\begin{subfigure}[b]{0.2\textwidth}
\centering
\raisebox{1.5cm}{
\begin{tikzpicture}[node distance = 1.5em, 
baseline={(1.base)}]
\node[root, vertex] (1) {};
\node[root, vertex, right = of 1] (2) {};
\node[root, vertex, right = of 2] (3) {};
\node[root, vertex, right = of 3] (4) {};
\end{tikzpicture}}
\caption{ID $\# 142$}
\end{subfigure}
\begin{subfigure}[b]{0.2\textwidth}
\centering
\raisebox{1.5cm}{
\begin{tikzpicture}[node distance = 1.5em, baseline={(1.base)}]
\node[root, vertex] (1) {};
\node[vertex, below = of 1] (2) {};
\node[root, vertex, right = of 1] (3) {};
\node[root, vertex, right = of 3] (4) {};
\draw (1)--(2);
\end{tikzpicture}}
\caption{ID $\# 130$}
\end{subfigure}
\begin{subfigure}[b]{0.2\textwidth}
\centering
\raisebox{1.5cm}{
\begin{tikzpicture}[node distance = 1.5em, baseline={(1.base)}]
\node[root, vertex] (1) {};
\node[vertex, below = of 1] (2) {};
\node[vertex, below = of 2] (3) {};
\node[root, vertex, right = of 1] (4) {};

\draw (1)--(2)--(3);
\end{tikzpicture}}
\caption{ID $\# 112$}
\end{subfigure}
\begin{subfigure}[b]{0.2\textwidth}
\centering
\raisebox{1.5cm}{
\begin{tikzpicture}[node distance = 1.5em, baseline={(1.base)}]
\node[root, vertex] (1) {};
\node[vertex, below = of 1] (2) {};
\node[root, vertex, right = of 1] (3) {};
\node[vertex, below = of 3] (4) {};

\draw (1)--(2)
(3)--(4);
\end{tikzpicture}}
\caption{ID $\# 106$}
\end{subfigure} \\ \vspace{1em}
\begin{subfigure}[b]{0.2\textwidth}
\centering
\raisebox{2cm}{
\begin{tikzpicture}[node distance = 1.5em and 1 em, baseline={(1.base)}]
\node[root, vertex] (1) {};
\node[vertex, below left = of 1] (2) {};
\node[vertex, below right = of 1] (3) {};
\node[root, vertex, right = 2.5em of 1] (4) {};

\draw (1)--(2)
(1)--(3);
\end{tikzpicture}}
\caption{ID $\# 95$}
\end{subfigure}
\begin{subfigure}[b]{0.2\textwidth}
\centering
\raisebox{2cm}{
\begin{tikzpicture}[node distance = 1.5em and 1 em, baseline={(1.base)}]
\node[root, vertex] (1) {};
\node[vertex, below left = of 1] (2) {};
\node[vertex, below right = of 1] (3) {};
\node[root, vertex, right = 2.5em of 1] (4) {};

\draw (1)--(2);
\draw (1)--(3) node[midway, right] {$-$};
\end{tikzpicture}}
\caption{ID $\# 131$}
\end{subfigure}
\begin{subfigure}[b]{0.18\textwidth}
\centering
\raisebox{2cm}{
\begin{tikzpicture}[node distance = 1.5em, baseline={(1.base)}]
\node[root, vertex] (1) {};
\node[vertex, below = of 1] (2) {};
\node[vertex, below = of 2] (3) {};
\node[vertex, below = of 3] (4) {};

\draw (1)--(2)--(3)--(4);
\end{tikzpicture}}
\caption{ID $\# 105$}
\end{subfigure}
\begin{subfigure}[b]{0.18\textwidth}
\centering
\raisebox{2cm}{
\begin{tikzpicture}[node distance = 1.5em and 1 em, baseline={(1.base)}]
\node[root, vertex] (1) {};
\node[vertex, below left = of 1] (2) {};
\node[vertex, below right = of 1] (3) {};
\node[vertex, below = of 2] (4) {};

\draw (1)--(2)--(4)
(1)--(3);
\end{tikzpicture}}
\caption{ID $\# 83$}
\end{subfigure}
\begin{subfigure}[b]{0.2\textwidth}
\centering
\raisebox{2cm}{
\begin{tikzpicture}[node distance = 1.5em and 1 em, baseline={(1.base)}]
\node[root, vertex] (1) {};
\node[vertex, below left = of 1] (2) {};
\node[vertex, below right = of 1] (3) {};
\node[vertex, below = of 2] (4) {};

\draw (1)--(2)--(4);
\draw (1)--(3) node[midway, right] {$-$};
\end{tikzpicture}}
\caption{ID $\# 108$}
\end{subfigure} \\ \vspace{1em}
\begin{subfigure}[b]{0.2\textwidth}
\centering
\raisebox{1.5cm}{
\begin{tikzpicture}[node distance = 1.5em and 1 em, baseline={(1.base)}]
\node[root, vertex] (1) {};
\node[vertex, below = of 1] (2) {};
\node[vertex, below left = of 2] (3) {};
\node[vertex, below right = of 2] (4) {};

\draw (1)--(2)--(3);
\draw (2)--(4);
\end{tikzpicture}}
\caption{ID $\# 75$}
\end{subfigure}
\begin{subfigure}[b]{0.2\textwidth}
\centering
\raisebox{1.5cm}{
\begin{tikzpicture}[node distance = 1.5em and 1 em, baseline={(1.base)}]
\node[root, vertex] (1) {};
\node[vertex, below = of 1] (2) {};
\node[vertex, below left = of 2] (3) {};
\node[vertex, below right = of 2] (4) {};

\draw (1)--(2)--(3);
\draw (2)--(4) node[midway, right] {$-$};
\end{tikzpicture}}
\caption{ID $\# 114$}
\end{subfigure}
\begin{subfigure}[b]{0.2\textwidth}
\centering
\raisebox{1.5cm}{
\begin{tikzpicture}[node distance = 1.5em and 1 em, baseline={(1.base)}]
\node[root, vertex] (1) {};
\node[vertex] (2) [below = of 1] {};
\node[vertex, below left = of 1] (3) {};
\node[vertex] (4) [below right = of 1] {};
\draw (1)--(2)
(1)--(3)
(1)--(4);
\end{tikzpicture}}
\caption{ID $\# 74$}
\end{subfigure}
\begin{subfigure}[b]{0.2\textwidth}
\centering
\raisebox{1.5cm}{
\begin{tikzpicture}[node distance = 1.5em and 1 em, baseline={(1.base)}]
\node[root, vertex] (1) {};
\node[vertex] (2) [below = of 1] {};
\node[vertex, below left = of 1] (3) {};
\node[vertex] (4) [below right = of 1] {};
\draw (1)--(2)
(1)--(3)
(1)--(4) node[midway, right] {$-$};
\end{tikzpicture}}
\caption{ID $\# 96$}
\end{subfigure}
\caption{Representatives of $\SF_4/\!\!\sim$. Numbers are ID's by 
{\O}bro.}\label{fig_SF4}
\end{figure}
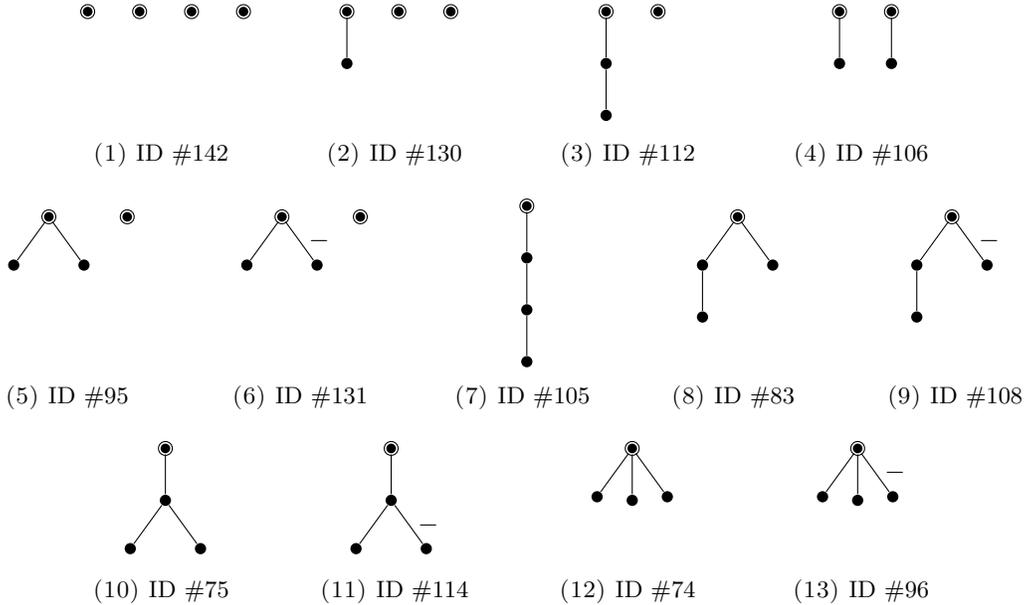

\noindent
\end{example}

Higashitani and Kurimoto~\cite{HigashitaniKurimoto20} provide another 
equivalence relation $\approx$ on the set of signed rooted forests which is 
used to consider the diffeomorphism classes in Fano Bott manifolds.
The equivalence relation $\approx$ is 
induced from the relation $\sim$ by neglecting signs on the edges incident on the roots. 
Using this relation, we recall the following. 
\begin{theorem}[{\cite[Theorem~1.8 and Remark~6.4]{HigashitaniKurimoto20}}]
The diffeomorphism classes in Fano Bott manifolds of complex dimension $n$ 
bijectively correspond to $\SF_n/\!\!\approx$. 
\end{theorem}

We say that a Bott manifold~$\Bt$ is \emph{indecomposable} if it is not isomorphic to a product of lower dimensional Bott manifolds (as toric varieties). Otherwise, we say that $\Bt$ is \emph{decomposable}.\footnote{The notion of indecomposability can have different meanings in different contexts. Especially, in~\cite{CMO}, they consider another family of smooth manifolds, called \emph{real Bott manifolds}, and say that a real Bott manifold is indecomposable if it is not \emph{diffeomorphic} to a product of lower dimensional real Bott manifolds.}


\begin{corollary}\label{cor_diffeom_FB}
The diffeomorphism classes of indecomposable Fano Bott manifolds of 
complex dimension $n$ bijectively correspond to $\SF_{n-1}/\!\!\sim$. 
\end{corollary}
\begin{proof}
We first notice that by Theorem~\ref{thm_isom_FB}, a Fano Bott 
manifold is indecomposable if and only if the corresponding signed rooted 
forest is a signed rooted tree, that is, it has only one root vertex.
Since the equivalence relation $\approx$ is induced from the relation $\sim$ by neglecting the signs on the edges incident on the root, we obtain the desired bijection by erasing the root vertex.
\end{proof}

\begin{example}
In Figure~\ref{fig_SF4}, three pairs ((5),(6)), ((8),(9)), 
((12),(13)) are diffeomorphic to each other but (10) and (11) are not 
diffeomorphic to each other. 
Considering signed rooted trees in Figure~\ref{fig_SF4}, we obtain the five 
equivalence classes 
\[
\{ (7), (8), (9), (10), (11), (12), (13)\}/\!\!\approx ~~= \{[(7)], [(8)], 
[(10)], [(11)], [(12)] \}. 
\]
By erasing the root vertex, each of which is associated to an element in 
$\SF_3/\!\!\sim$.
\[
[(7)] \leftrightarrow [(\T_1,\s_1)], \quad 
[(8)] \leftrightarrow [(\T_8,\s_8)], \quad
[(10)] \leftrightarrow [(\T_5,\s_5)],\quad
[(11)] \leftrightarrow [(\T_6,\s_6)],\quad
[(12)] \leftrightarrow [(\T_{10},\s_{10})].
\]
\end{example}
\section{Counting signed rooted forests in terms of signed rooted 
trees}\label{sec_counting}
We denote by $\ST_n/\!\!\sim$ the set of signed rooted trees in 
$\SF_n/\!\!\sim$. We set $\f_n=|\ST_n/\!\!\sim|$ and $f_n=|\SF_n/\!\!\sim|$. 
Now we let $\MT(x)$ and $\MF(x)$ be the generating functions of the sequences 
$\{\f_n\}$ and $\{f_n\}$, respectively, that is,
\[
\MT(x)=\sum_{n=1}^\infty \f_nx^n\qquad\text{ and }\qquad 
\MF(x)=1+\sum_{n=1}^\infty 
f_nx^n.
\] In this section, we compute the generating functions $\MT(x)$ and $\MF(x)$, 
and study their relations.

\begin{proposition} \label{prop:3-1}
The generating function $\MF(x)$ satisfies
\[
\MF(x)=\displaystyle{{\prod_{k=1}^\infty (1-x^k)^{-\f_k}}}.
\]
\end{proposition}

\begin{proof}
Note that from the generalized binomial theorem, for any positive integer $m$, 
we have
\begin{equation*} 
\begin{split}
(1-x)^{-m}&=\sum_{p=0}^\infty \binom{-m}{p}(-x)^p\\
&=\sum_{p=0}^\infty \frac{(-m)(-m-1)\cdots (-m-(p-1))}{p!}(-x)^p\\
&=\sum_{p=0}^\infty\binom{m-1+p}{p}x^p.
\end{split}
\end{equation*}
Then, since 
\begin{equation} \label{eq:pk}
{(1-x^k)^{-\f_k}}=\sum_{p_k=0}^\infty \binom{\f_k-1+p_k}{p_k}x^{kp_k},
\end{equation}
the coefficient of $x^n$ in the product 
$\displaystyle{{\prod_{k=1}^\infty (1-x^k)^{-\f_k}}}$ is given by 
\begin{equation} \label{eq:p1_pn}
\sum_{(p_1,\dots,p_n)}\prod_{k=1}^n\binom{\f_k-1+p_k}{p_k},
\end{equation}
where $(p_1,\dots,p_n)$ runs over all $n$-tuples of nonnegative integers with 
$\sum_{k=1}^n kp_k=n$. Here $\binom{\f_k-1+p_k}{p_k}$ is the number of 
signed rooted forests with $p_k$ components such that each component is a 
signed rooted tree with $k$ vertices by \eqref{eq:pk} and the sum in \eqref{eq:p1_pn} counts all 
decompositions of elements in $\SF_n/\!\!\sim$ into connected components, so 
the proposition follows. 
\end{proof}
The following is an easy consequence of the above proposition.
\begin{corollary} \label{lemm:3-5}
The generating functions $\MF(x)$ and $\MT(x)$ satisfy
\[
\MF(x)=\exp\left(\sum_{n=1}^\infty \frac{\MT(x^n)}{n}\right).
\] 
\end{corollary}

\begin{proof}
Taking logarithm on both sides of \[
\MF(x)=\prod_{k=1}^\infty(1-x^k)^{-\f_k},
\] we obtain
\[
\log \MF(x)=-\sum_{k=1}^\infty 
\f_k\log(1-x^k)=\sum_{k=1}^n\sum_{n=1}^\infty 
\f_k\frac{x^{kn}}{n}=\sum_{n=1}^\infty \frac{\MT(x^n)}{n},
\]
which implies the corollary. 
\end{proof}

\begin{lemma}[cf. (1) in \cite{HaUh53}] \label{lemm:3-4}
The generating functions $\MF(x)$ and $\MT(x)$ satisfy
\[
\MT(x)=\frac{x}{2}\left(\MF(x^2)+\MF(x)^2\right).
\]
\end{lemma}

\begin{proof}

We set 
\[
\ST/\!\!\sim=\bigsqcup_{n=1}^\infty \ST_n/\!\!\sim\qquad \text{ and }\qquad
\SF/\!\!\sim=\bigsqcup_{n=0}^\infty \SF_n/\!\!\sim,
\]
where $\SF_0/\!\!\sim$ is understood to be the empty set.
Given an unordered pair $\{A, B\}$ of $\SF/\!\!\sim$, we obtain an element $AB$ of $\ST/\!\!\sim$ by joining the roots of $A$ and $B$ to a new root $v$ and assign all the new edges joining the roots of $A$ to $v$, say $+$ sign, and all the new edges joining the roots of $B$ to $v$, say $-$ sign. We may assign $-$ sign to the former and $+$ sign to the latter. In any case, $AB$ is well-defined in $\ST/\!\!\sim$. 
Conversely, given an element $T$ of $\ST/\!\!\sim$, there is a unique unordered 
pair $\{A, B\}$ of $\SF/\!\!\sim$ such that $AB=T$. 

This implies the lemma. Indeed, $\MF(x)^2$ counts unordered pairs $\{A,B\}$ 
twice when $A$ and $B$ are different but once when $A=B$. This is why we add 
$\MF(x^2)$ in the formula. Multiplication by $x$ corresponds to the new vertex 
$v$. 
\end{proof}

Combining Corollary~\ref{lemm:3-5} and Lemma~\ref{lemm:3-4}, we obtain the 
following functional equation mentioned in the introduction.

\begin{corollary}[cf. (3) in \cite{HaUh53}]\label{cor_generating_ftn_of_Fn}
The generating function $\MF(x)$ satisfies
\[
\MF(x)=\exp\left(\sum_{n=1}^\infty\frac{x^n}{2n}\left(\MF(x^{2n})+\MF(x^n)^2\right)\right).
\]
\end{corollary}

This functional equation determines $\MF(x)$. Using Lemma~\ref{lemm:3-4} and Corollary~\ref{cor_generating_ftn_of_Fn}, we obtain the following table.
\begin{table}[ht]
\centering
\begin{tabular}{c| r r r r r r r r r r} 
\toprule
$n$& 1&2&3&4&5&6&7&8&9&10\\ 
\midrule 
$t_n$ &1&1&3&7&21&60&189&595&1948&6455\\ \hline 
$f_n$& 1&2&5&13&37&111&345&1105&3624&12099\\ 
\bottomrule
\end{tabular} 
\medskip
\caption{The numbers of signed rooted trees and signed rooted forests} 
\label{table_tn_fn}
\end{table}

The numbers $t_n$ and $f_n$ in Table~\ref{table_tn_fn} satisfy $f_n<2\f_n<4f_{n-1}$ for $n\le 10$ and the sequences $\{t_n/t_{n-1}\}_{n=2}^{10}$ and $\{f_{n}/f_{n-1}\}_{n=2}^{10}$ are both increasing.

\begin{Question} Are the sequences $t_n/t_{n-1}$ and $f_n/f_{n-1}$ increasing and bounded above by $4$?
\end{Question}

\section{Rooted triangular cacti} \label{sec_cacti}

The formula in Corollary~\ref{cor_generating_ftn_of_Fn} also holds for the generating function of the number of rooted triangular cacti with $2n+1$ vertices and $n$ triangles. 
In this section, we give a bijective correspondence between the equivalence classes of signed rooted forests $\SF_n/\!\!\sim$ and rooted 
triangular cacti with $2n+1$ vertices and $n$ triangles.

\begin{definition}
A \emph{cactus} (or a \emph{cactus tree}) is a connected graph in which any 
two simple cycles have at most one vertex in common, equivalently, no line 
lies 
on more than one cycle.
A \emph{triangular cactus} (or a \emph{$3$-cactus}) is a cactus such that 
every cycle has length three. A \emph{rooted triangular cactus} is a 
triangular 
cactus having a root vertex. 
\end{definition} 
We sometimes call a $3$-cycle in a $3$-cactus a \emph{triangle}.
In Figure~\ref{fig_cacti_4_triangles}, we present rooted $3$-cacti having nine 
vertices and four triangles.
The sequence of numbers of rooted $3$-cacti with $2n+1$ vertices and $n$ 
triangles is \href{https://oeis.org/A003080}{Sequence~A003080} 
in~\cite{OEIS}. 
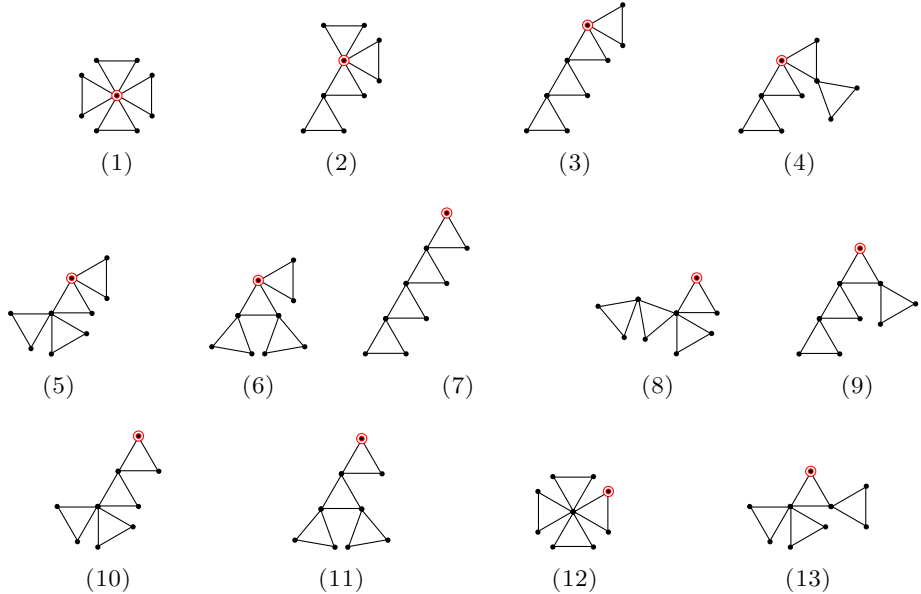
\begin{figure}[h]
\tikzset{trinode/.style = {draw, regular polygon, regular polygon sides = 3, inner sep = 3pt, rotate = 0}}
\tikzset{vertex/.style = {circle, fill, inner sep = 0.7pt}}
\tikzset{root/.style = {circle, double, draw, inner sep = 1pt, red}}
\begin{subfigure}[b]{0.2\textwidth}
\centering
\begin{tikzpicture}
\node[trinode] (1) {};
\node[trinode, rotate = -30, anchor = corner 2] at (1.corner 1) (2) {};
\node[trinode, rotate = 60, anchor = corner 2] at (1.corner 1) (3) {};
\node[trinode, rotate = 150, anchor = corner 2] at (1.corner 1) (4) {};

\foreach \x in {1,2,3,4}{
\node [vertex] at (\x.corner 1) {};
\node [vertex] at (\x.corner 2) {};
\node [vertex] at (\x.corner 3) {};
}

\node[root] at (1.corner 1) {};
\end{tikzpicture}
\caption{}
\end{subfigure}%
\begin{subfigure}[b]{0.2\textwidth}
\centering
\begin{tikzpicture}
\node[trinode, rotate = 90] (1) {};
\node[trinode, anchor = corner 2, rotate = 240] at (1.corner 1) (2) {};
\node[trinode, anchor = corner 1, rotate = 0] at (2.corner 3) (3) {};
\node[trinode, anchor = corner 2, rotate = 60] at (1.corner 1) (4) {};

\foreach \x in {1,2,3,4}{
\node [vertex] at (\x.corner 1) {};
\node [vertex] at (\x.corner 2) {};
\node [vertex] at (\x.corner 3) {};
}

\node[root] at (1.corner 1) {};
\end{tikzpicture}
\caption{}
\end{subfigure}
\begin{subfigure}[b]{0.2\textwidth}
\centering
\begin{tikzpicture}

\node[trinode, rotate = 90] (1) {};
\node[trinode, anchor = corner 1] at (1.corner 1) (2) {};
\node[trinode, anchor = corner 1] at (2.corner 2) (3) {};
\node[trinode, anchor = corner 1] at (3.corner 2) (4) {};

\foreach \x in {1,2,3,4}{
\node [vertex] at (\x.corner 1) {};
\node [vertex] at (\x.corner 2) {};
\node [vertex] at (\x.corner 3) {};
}

\node[root] at (2.corner 1) {};
\end{tikzpicture}
\caption{}
\end{subfigure}%
\begin{subfigure}[b]{0.2\textwidth}
\centering
\begin{tikzpicture}
\node[trinode, rotate = 0] (1) {};
\node[trinode, anchor = corner 1, rotate = 0] at (1.corner 2) (2) {};
\node[trinode, anchor = corner 1, rotate = 90] at (1.corner 1) (3) {};
\node[trinode, anchor = corner 1, rotate = 50] at (3.corner 2) (4) {};

\foreach \x in {1,2,3,4}{
\node [vertex] at (\x.corner 1) {};
\node [vertex] at (\x.corner 2) {};
\node [vertex] at (\x.corner 3) {};
}

\node[root] at (1.corner 1) {};
\end{tikzpicture}
\caption{}
\end{subfigure}
\\ \vspace{1em}

\begin{subfigure}[b]{0.17\textwidth}
\centering
\begin{tikzpicture}
\node[trinode, rotate = 90] (1) {};
\node[trinode, anchor = corner 1,rotate = 0] at (1.corner 1) (2) {};
\node[trinode, rotate = 30, anchor = corner 1] at (2.corner 2) (3) {};
\node[trinode, rotate = -60, anchor = corner 1] at (2.corner 2) (4) {};

\foreach \x in {1,2,3,4}{
\node [vertex] at (\x.corner 1) {};
\node [vertex] at (\x.corner 2) {};
\node [vertex] at (\x.corner 3) {};
}

\node[root] at (2.corner 1) {};
\end{tikzpicture}
\caption{}
\end{subfigure}
\begin{subfigure}[b]{0.17\textwidth}
\centering
\begin{tikzpicture}
\node[trinode, rotate = 90] (1) {};
\node[trinode, anchor = corner 1, rotate = 0] at (1.corner 1) (2) {};
\node[trinode, rotate = 10, anchor = corner 1] at (2.corner 3) (3) {};
\node[trinode, rotate = -10, anchor = corner 1] at (2.corner 2) (4) {};

\foreach \x in {1,2,3,4}{
\node [vertex] at (\x.corner 1) {};
\node [vertex] at (\x.corner 2) {};
\node [vertex] at (\x.corner 3) {};
}

\node[root] at (2.corner 1) {};
\end{tikzpicture}
\caption{}
\end{subfigure}
\begin{subfigure}[b]{0.17\textwidth}
\begin{tikzpicture}
\node[trinode] (1) {};
\node[trinode, anchor = corner 1] at (1.corner 2) (2) {};
\node[trinode, anchor = corner 1] at (2.corner 2) (3) {};
\node[trinode, anchor = corner 1] at (3.corner 2) (4) {};

\foreach \x in {1,2,3,4}{
\node [vertex] at (\x.corner 1) {};
\node [vertex] at (\x.corner 2) {};
\node [vertex] at (\x.corner 3) {};
}

\node[root] at (1.corner 1) {};
\end{tikzpicture}
\caption{}
\end{subfigure}
\begin{subfigure}[b]{0.17\textwidth}
\centering
\begin{tikzpicture}
\node[trinode] (1) {};
\node[trinode, rotate = 30, anchor = corner 1] at (1.corner 2) (2) {};
\node[trinode, rotate = -80, anchor = corner 1] at (1.corner 2) (3) {};
\node[trinode, rotate = -50, anchor = corner 1] at (3.corner 2) (4) {};

\foreach \x in {1,2,3,4}{
\node [vertex] at (\x.corner 1) {};
\node [vertex] at (\x.corner 2) {};
\node [vertex] at (\x.corner 3) {};
}

\node[root] at (1.corner 1) {};
\end{tikzpicture}
\caption{}
\end{subfigure}
\begin{subfigure}[b]{0.17\textwidth}
\centering
\begin{tikzpicture}
\node[trinode] (1) {};
\node[trinode, rotate = 30, anchor = corner 1] at (1.corner 3) (2) {};
\node[trinode, rotate = 0, anchor = corner 1] at (1.corner 2) (3) {};
\node[trinode, rotate = 0, anchor = corner 1] at (3.corner 2) (4) {};

\foreach \x in {1,2,3,4}{
	\node [vertex] at (\x.corner 1) {};
	\node [vertex] at (\x.corner 2) {};
	\node [vertex] at (\x.corner 3) {};
}

\node[root] at (1.corner 1) {};
\end{tikzpicture}
\caption{}
\end{subfigure}
\\ \vspace{1em}

\begin{subfigure}[b]{0.2\textwidth}
\centering
\begin{tikzpicture}
\node[trinode, rotate = 120] (1) {};
\node[trinode, anchor = corner 1,rotate = 0] at (1.corner 1) (2) {};
\node[trinode, rotate = 30, anchor = corner 1] at (2.corner 2) (3) {};
\node[trinode, rotate = -60, anchor = corner 1] at (2.corner 2) (4) {};

\foreach \x in {1,2,3,4}{
	\node [vertex] at (\x.corner 1) {};
	\node [vertex] at (\x.corner 2) {};
	\node [vertex] at (\x.corner 3) {};
}

\node[root] at (1.corner 3) {};
\end{tikzpicture}
\caption{}
\end{subfigure}
\begin{subfigure}[b]{0.20\textwidth}
\centering
\begin{tikzpicture}
\node[trinode, rotate = 120] (1) {};
\node[trinode, anchor = corner 1, rotate = 0] at (1.corner 1) (2) {};
\node[trinode, rotate = 10, anchor = corner 1] at (2.corner 3) (3) {};
\node[trinode, rotate = -10, anchor = corner 1] at (2.corner 2) (4) {};

\foreach \x in {1,2,3,4}{
	\node [vertex] at (\x.corner 1) {};
	\node [vertex] at (\x.corner 2) {};
	\node [vertex] at (\x.corner 3) {};
}

\node[root] at (1.corner 3) {};
\end{tikzpicture}
\caption{}
\end{subfigure}
\begin{subfigure}[b]{0.2\textwidth}
\centering
\begin{tikzpicture}
\node[trinode] (1) {};
\node[trinode, rotate = -30, anchor = corner 2] at (1.corner 1) (2) {};
\node[trinode, rotate = 60, anchor = corner 2] at (1.corner 1) (3) {};
\node[trinode, rotate = 150, anchor = corner 2] at (1.corner 1) (4) {};

\foreach \x in {1,2,3,4}{
	\node [vertex] at (\x.corner 1) {};
	\node [vertex] at (\x.corner 2) {};
	\node [vertex] at (\x.corner 3) {};
}

\node[root] at (2.corner 1) {};
\end{tikzpicture}
\caption{}
\end{subfigure}
\begin{subfigure}[b]{0.2\textwidth}
\centering
\begin{tikzpicture}

\node[trinode, anchor = corner 1,rotate = 0] (2) {};
\node[trinode, rotate = 90, anchor = corner 1] at (2.corner 3) (1) {};
\node[trinode, rotate = 30, anchor = corner 1] at (2.corner 2) (3) {};
\node[trinode, rotate = -60, anchor = corner 1] at (2.corner 2) (4) {};

\foreach \x in {1,2,3,4}{
	\node [vertex] at (\x.corner 1) {};
	\node [vertex] at (\x.corner 2) {};
	\node [vertex] at (\x.corner 3) {};
}

\node[root] at (2.corner 1) {};
\end{tikzpicture}
\caption{}
\end{subfigure}

\caption{Rooted triangular cacti}\label{fig_cacti_4_triangles}
\end{figure}
\begin{remark}
A cactus is also called a \textit{Husimi tree} (see~\cite{HaNo53, HaUh53}).
\end{remark}
\begin{proposition}
There is a bijective correspondence between $\SF_n/\!\!\sim$ and the set of 
rooted $3$-cacti with $2n+1$ vertices and $n$ triangles. 
\end{proposition}
\begin{proof}
Let $(\mathcal{T},\mathsf{s})$ be a signed rooted forest.
For each root vertex of $(\mathcal{T},\mathsf{s})$, we draw a triangle and decorate the top vertex of the triangle with a double circle to indicate the root of the triangular cacti. 
For each child of the root of $\mathcal{T}$, 
we draw a triangle as follows. If the sign of the edge incident on the root is positive, we attach the new triangle to the left bottom vertex; if the sign is negative, we attach the new triangle to the right bottom vertex. Continuing this process to each child vertex, we get a bunch of rooted 
triangular cacti. Finally, we merge all the root vertices of rooted triangular 
cacti to one root vertex so we obtain one rooted triangular cacti.
See Figure~\ref{fig-cacti-forest}. Obviously, the rooted triangular cacti 
corresponding to $(\mathcal{T},\mathsf{s})$ and $r_{i}(\mathcal{T},\mathsf{s})$ 
are isomorphic to each other. This proves the proposition. 
\end{proof}
\begin{figure}[h]
\tikzset{trinode/.style = {draw, regular polygon, regular polygon sides = 3, 
inner sep = 5pt, rotate = 0}}
\tikzset{Vertex/.style = {circle, fill, inner sep = 0.7pt}}
\tikzset{Root/.style = {circle, double, draw, inner sep = 1pt, red}}
\begin{tikzpicture}
\begin{scope}[opacity=0.3]
\node[trinode, rotate = 120] (1) {};
\node[trinode, anchor = corner 1,rotate = 0] at (1.corner 1) (2) {};
\node[trinode, rotate = 30, anchor = corner 1] at (2.corner 2) (3) {};
\node[trinode, rotate = -60, anchor = corner 1] at (2.corner 2) (4) {};

\foreach \x in {1,2,3,4}{
	\node [Vertex] at (\x.corner 1) {};
	\node [Vertex] at (\x.corner 2) {};
	\node [Vertex] at (\x.corner 3) {};
}

\node[Root] at (1.corner 3) {};
\end{scope}
\foreach \x in {1,2,3,4}{
\node[vertex] at (\x.center) (v\x) {};
}

\node[root] at (v1) {};
\draw (v1)--(v2)
(v2)--(v3) 
(v2)--(v4);
\end{tikzpicture}
~\hspace{2cm}%
\begin{tikzpicture}
\begin{scope}[opacity=0.3]
\node[trinode, rotate = 120] (1) {};
\node[trinode, anchor = corner 1, rotate = 0] at (1.corner 1) (2) {};
\node[trinode, rotate = 10, anchor = corner 1] at (2.corner 3) (3) {};
\node[trinode, rotate = -10, anchor = corner 1] at (2.corner 2) (4) {};

\foreach \x in {1,2,3,4}{
	\node [Vertex] at (\x.corner 1) {};
	\node [Vertex] at (\x.corner 2) {};
	\node [Vertex] at (\x.corner 3) {};
}

\node[Root] at (1.corner 3) {};
\end{scope}

\foreach \x in {1,2,3,4}{
\node[vertex] at (\x.center) (v\x) {};
}

\node[root] at (v1) {};
\draw (v1)--(v2)
(v2)--(v3) node[midway, right] {$-$}
(v2)--(v4);
\end{tikzpicture}
~\hspace{2cm}%
\begin{tikzpicture}
\begin{scope}[opacity = 0.3]
\node[trinode, rotate = 90] (1) {};
\node[trinode, anchor = corner 1, rotate = 0] at (1.corner 1) (2) {};
\node[trinode, rotate = 10, anchor = corner 1] at (2.corner 3) (3) {};
\node[trinode, rotate = -10, anchor = corner 1] at (2.corner 2) (4) {};

\foreach \x in {1,2,3,4}{
	\node [Vertex] at (\x.corner 1) {};
	\node [Vertex] at (\x.corner 2) {};
	\node [Vertex] at (\x.corner 3) {};
}

\node[Root] at (2.corner 1) {};
\end{scope}
\foreach \x in {1,2,3,4}{
\node[vertex] at (\x.center) (v\x) {};
}
\node[root] at (v1) {};
\node[root] at (v2) {};

\draw (v2)--(v3) node[midway, right] {$-$}
(v2)--(v4);

\end{tikzpicture}\caption{Construction of triangular cacti from signed rooted 
forests}\label{fig-cacti-forest}
\end{figure}
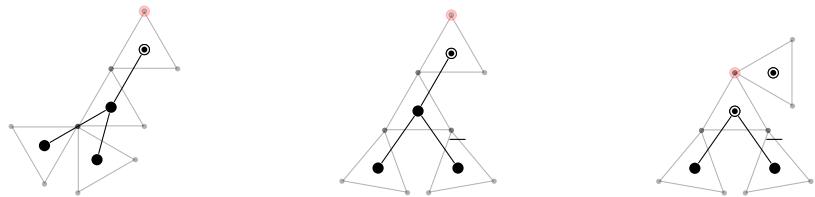


\begin{thebibliography}{10}
	
	\bibitem{Batyrev}
	Victor~V. Batyrev, \emph{On the classification of toric {F}ano {$4$}-folds},
	vol.~94, 1999, Algebraic geometry, 9, pp.~1021--1050. 
	
	\bibitem{BCT19_Kahler}
	Charles~P. Boyer, David M.~J. Calderbank, and Christina~W. T\o~nnesen Friedman,
	\emph{The {K}\"{a}hler geometry of {B}ott manifolds}, Adv. Math. \textbf{350}
	(2019), 1--62. 
	
	\bibitem{GRDB}
	Gavin Brown and Alexander Kasprzyk, \emph{Graded ring database},
	\href{http://www.grdb.co.uk/forms/toricsmooth}{\texttt{http://www.grdb.co.uk/forms/toricsmooth}} (accessed \today).
	
	\bibitem{CLMP}
	Yunhyung Cho, Eunjeong Lee, Mikiya Masuda, and Seonjeong Park, \emph{Unique
		toric structure on a {F}ano {B}ott manifold}, arXiv preprint
	arXiv:2005.02740v1 (2020).
	
	\bibitem{ChoiMasudaMurai}
	Suyoung Choi, Mikiya Masuda, and Satoshi Murai, \emph{Invariance of
		{P}ontrjagin classes for {B}ott manifolds}, Algebr. Geom. Topol. \textbf{15}
	(2015), no.~2, 965--986. 
	
	\bibitem{CMO}
	Suyoung Choi, Mikiya Masuda, and Sang-il Oum, \emph{Classification of real
		{B}ott manifolds and acyclic digraphs}, Trans. Amer. Math. Soc. \textbf{369}
	(2017), no.~4, 2987--3011. 
	
	\bibitem{ChoiSuh11}
	Suyoung Choi and Dong~Youp Suh, \emph{Properties of {B}ott manifolds and
		cohomological rigidity}, Algebr. Geom. Topol. \textbf{11} (2011), no.~2,
	1053--1076. 
	
	\bibitem{GK94Bott}
	Michael Grossberg and Yael Karshon, \emph{Bott towers, complete integrability,
		and the extended character of representations}, Duke Math. J. \textbf{76}
	(1994), no.~1, 23--58. 
	
	\bibitem{HaNo53}
	Frank Harary and Robert~Z. Norman, \emph{The dissimilarity characteristic of
		{H}usimi trees}, Ann. of Math. (2) \textbf{58} (1953), 134--141. 
	
	\bibitem{HaUh53}
	Frank Harary and George~E. Uhlenbeck, \emph{On the number of {H}usimi trees.
		{I}}, Proc. Nat. Acad. Sci. U.S.A. \textbf{39} (1953), 315--322. 
	
	\bibitem{HigashitaniKurimoto20}
	Akihiro Higashitani and Kazuki Kurimoto, \emph{Cohomological rigidity for
		{F}ano {B}ott manifolds}, arXiv preprint arXiv:2008.05811v1 (2020).
	
	\bibitem{LMP21}
	Eunjeong Lee, Mikiya Masuda, and Seonjeong Park, \emph{Toric {R}ichardson
		varieties of {C}atalan type and {W}edderburn--{E}therington numbers}, arXiv
	preprint arXiv:2105.12274v1 (2021).
	
	\bibitem{OEIS}
	{N. J. A. Sloane, editor}, \emph{\textup{The {O}n-{L}ine {E}ncyclopedia of
			{I}nteger {S}equences ({OEIS})}}, published electronically at
	\href{https://oeis.org}{\texttt{https://oeis.org}} (accessed \today).
	
	\bibitem{Chary18}
	B.~Narasimha~Chary, \emph{On {M}ori cone of {B}ott towers}, J. Algebra
	\textbf{507} (2018), 467--501. 
	
	\bibitem{Obro07}
	Mikkel {\O}bro, \emph{An algorithm for the classification of smooth {F}ano
		polytopes}, arXiv preprint arXiv:0704.0049v1 (2007).
	
	\bibitem{Suyama20}
	Yusuke Suyama, \emph{Fano generalized {B}ott manifolds}, Manuscripta Math.
	\textbf{163} (2020), no.~3-4, 427--435.
	
\end{thebibliography}
\providecommand{\bysame}{\leavevmode\hbox to3em{\hrulefill}\thinspace}
\providecommand{\MR}{\relax\ifhmode\unskip\space\fi MR }
\providecommand{\MRhref}[2]{%
	\href{http://www.ams.org/mathscinet-getitem?mr=#1}{#2}
}
\providecommand{\href}[2]{#2}

\end{document}